\documentclass[11pt]{article}

\usepackage{amsmath,amssymb,amsthm,mathtools,mathrsfs}
\usepackage{geometry}
\usepackage{xcolor}
\usepackage{enumitem}
\usepackage{tikz-cd}
\usepackage{cite}
\usepackage[hidelinks]{hyperref}

\title{Schottky versus super Schottky in genus 4}

\author{
Ron Donagi\thanks{Departments of Mathematics and Physics, University of Pennsylvania,
Philadelphia, PA 19104, USA.}
\and
Simone Noja\thanks{Dipartimento di Matematica, Universit\`a degli Studi di
Bari Aldo Moro, Via E.~Orabona 4, 70125 Bari, Italy.}}

\date{}

\newtheorem{theorem}{Theorem}[section]
\newtheorem{proposition}[theorem]{Proposition}
\newtheorem{lemma}[theorem]{Lemma}

\newtheorem{definition}[theorem]{Definition}

\newcommand{\OO}{\mathcal O}
\newcommand{\MM}{\mathcal M}
\newcommand{\FM}{\mathfrak M}
\newcommand{\Aa}{\mathcal A}
\newcommand{\SSS}{\mathcal S}

\newcommand{\Sym}{\operatorname{Sym}}
\newcommand{\rk}{\operatorname{rk}}

\begin{document}

\maketitle
\tableofcontents

\begin{abstract}
The study of super Riemann surfaces and their moduli led Witten and Felder, Kazhdan and Polishchuk to ask what is the smallest
$d$ such that the $d$-th power of the Schottky ideal is contained in the super Schottky ideal.
Felder, Kazhdan and Polishchuk proved that $d=g$ in odd genus $g\geq5$ and that $d\in\{g-1,g\}$ in even genus $g\geq4$; Y.~Shen subsequently proved that $d=g$ in even genus $g\geq6$.  Thus the only remaining genus with nontrivial Schottky ideal was $g=4$.
Here we close this gap by showing that in genus $4$ too, the smallest power of the Schottky ideal contained in the super
Schottky ideal is $d=g=4$. The question is equivalent to one concerning the codifferential of the super period map.
{The quadratic odd contribution sends a conormal vector to a bivector,
represented by a skew-symmetric $(2g-2)\times(2g-2)$ matrix.
The question is to find such a vector for which, at a generic point,
this matrix has maximal rank $2g-2$, or $6$ in our case.}
Our computation uses the standard exact sequences on $C\times C$ attached to
the sheaves $\OO(a,b,c)$, together with a degeneration to a vanishing theta-null and some facts about the Szeg\H{o} kernel.
At a curve with a vanishing theta-null, a regularized version of the codifferential can be described by an explicit multiplication map,
allowing an easy computation of the rank. This rank turns out to be $4$.
{A first-order calculation on an explicit deformation of a cyclic trigonal
vanishing theta-null differentiates the global Gaussian-map identity for the
regularized conormal bivector. Including the variation of the Gaussian map
gives a nonzero first normal symbol.}
The resulting first-order form is nonzero on the null space of the limiting skew-symmetric matrix,
so the rank jumps to $6$ on nearby curves, showing that $d=4$.
\end{abstract}

\section{Introduction}

We work over $\mathbb C$, with $g\geq2$. Let
\[
p:\MM_g\to\Aa_g
\]
be the classical period map. Write $\FM_g^{+,0}$ for the open substack
of even supermoduli whose reduced spin curves $(C,L)$ satisfy
$h^0(C,L)=0$, and let
\[
\widetilde p:\FM_g^{+,0}\longrightarrow\Aa_g
\]
be the super period map on this domain. When period coordinates are
needed, we pass to a symplectically marked cover. Across the theta-null
divisor the map is in general meromorphic.
Write $I_g$ for the classical Schottky ideal and $\widetilde I_g$ for
the ideal of the schematic image of this restricted super period map:
locally, $f\in\widetilde I_g$ precisely when $\widetilde p^{\,*}f=0$.
These definitions can equivalently be made on the marked cover.

Since the target is purely bosonic, the super period map factors through
\[
\widetilde p:\FM_g^{+,0}/\iota\longrightarrow\Aa_g.
\]
Here $\iota$ is the tautological parity involution, acting as the identity
on even coordinates and as $-1$ on odd coordinates, and $/\iota$ denotes
the bosonic quotient of \cite{CV}: its local functions are the invariant,
that is, even, functions. It is not a quotient stack by an additional
external $\boldsymbol\mu_2$-action.

In order to understand the image of the super period map, Witten asked what was the smallest integer $d$ such that
\[
(I_g)^d\subset \widetilde I_g.
\]
Since $\FM_g$ has odd dimension $2g-2$ (and $\Aa_g$ is purely bosonic), one always has
$(I_g)^g\subset \widetilde I_g,$
or in other words, $d \leq g$.
This question was studied in depth by Felder, Kazhdan and Polishchuk in {\cite[Theorem~8.3]{FKP}}.
They proved that:
\begin{itemize}[leftmargin=2em]
\item if $g\geq5$ is odd, then $d=g$;
\item if $g\geq4$ is even, then $d\ge g-1$.
\end{itemize}
Their results leave open the question of whether the minimal $d$ equals $g$ or $g-1$ for even $g\ge 4$.
In a companion paper \cite{S}, Yuanyuan Shen settles most of this ambiguity:
she shows that the smallest $d$ equals $g$ for all even $g\ge 6$.
This leaves open the single case $g=4$. In this note we show that in this case as well, $d=g=4$.
{The genus-$4$ argument below is independent of the higher-genus result in \cite{S}.}

\begin{theorem}\label{thm:main}
In genus $4$, the smallest $d$ such that
$(I_4)^d\subset \widetilde I_4$
is $d=4.$
\end{theorem}

The proof uses the geometry of the sheaves
\[
\OO(a,b,c)=p_1^*K_C^{a/2}\otimes p_2^*K_C^{b/2}\otimes \OO_{C\times C}(c\Delta)
\]
on $C\times C$, where $C$ is a spin curve and $\Delta\subset C\times C$ is the diagonal.
These sheaves and their exact sequences control, in a uniform way,
classical analytic kernels on curves and intrinsic classes on moduli.
In particular, they give the Atiyah class of $\MM_g$ and the obstruction class of $\FM_g$;
see \cite{DW2}.
They also provide a convenient framework for the codifferential of the super period map,
as well as the related Szeg\H{o} kernel.

The codifferential ${d^*\widetilde{p}}$ of the super period map $\widetilde{p}$
can be interpreted as multiplication with the Szeg\H{o} kernel $s$.
It takes a conormal vector $q$ to the Jacobian locus
to an element $sq\in \wedge^2 V$, where
\[
V:=H^0(K_C^{\otimes \frac32}).
\]
The question boils down to computing the rank of $sq$,
which can be thought of as a skew-symmetric $(2g-2)\times (2g-2)$ matrix.
The claim, proved by Felder--Kazhdan--Polishchuk in odd genus and by Shen in even genus $g\ge 6$, and proved here for $g=4$,
is that generically this rank takes the largest possible value, namely $2g-2$.
This implies that the $(g-1)$-st power of the Schottky ideal is not contained in super Schottky,
so $d=g$.

The calculation at curves with a vanishing theta-null is classical in spirit.
Both the super period map and the Szeg\H{o} kernel blow up there,
so we work instead with a regularized version that remains regular.
Everything is encoded in the multiplication map
\[
\Sym^2 H^0(K_C)\otimes \wedge^2 H^0(K_C^{1/2})
\longrightarrow
\wedge^2 H^0(K_C^{3/2}).
\]
After rescaling the Szeg\H{o} kernel, its limiting value lies in
$\wedge^2 H^0(K_C^{1/2})$, while the conormal space is the space of canonical quadrics
\[
I_2(K_C):=\ker\!\left(
\operatorname{Sym}^2 H^0(K_C)\longrightarrow H^0(K_C^2)
\right).
\]
{For a nonhyperelliptic genus-$4$ curve,} this space is one-dimensional, generated by the unique canonical quadric.
Everything can then be written explicitly, and the image in
$\wedge^2 H^0(K_C^{3/2})$ is seen to have rank $4$.

To pass from rank $4$ to rank $6$, we work at a cyclic trigonal vanishing theta-null
and choose a deformation transverse to the theta-null divisor.  We compute the first variation
of the {regularized conormal bivector} on the two-dimensional null space of the limiting rank-$4$ skew-symmetric matrix.
\begingroup
We differentiate the global Gaussian-map identity for the regularized
bivector, using actual global first-order lifts of sections of $L^3$.
An explicit kernel calculation for the central Gaussian map detects the
resulting class in the normal quotient. A subsequent Serre-duality
calculation determines the polar coefficient of the regularized
Szeg\H{o} kernel and proves that the first normal symbol is nonzero.
The rank-jump lemma then gives rank $6$ on nearby fibers, and openness
and irreducibility give the generic statement.
\endgroup

\section*{Acknowledgments}

The authors benefited from helpful conversations with Yuanyuan Shen,
who was working on a parallel project in higher genus.
They are grateful to Sasha Polishchuk, who ran the paper through ChatGPT
and shared with us its very helpful feedback.
The authors also thank Sam Grushevsky, Nadia Ott, and Edward
Witten for helpful correspondence and discussions.

The research of R.D. was partially supported 
by NSF grant DMS--2401422, \emph{Geometry and Strings};
by NSF grant DMS--2244978, \emph{FRG: New Birational Invariants}; 
and by DFG--SFB 1624, \emph{Higher structures, moduli spaces and integrability}--506632645.

\section{Background on spin curves and the sheaves
\texorpdfstring{$\OO(a,b,c)$}{O(a,b,c)}}

\subsection{Spin curves and vanishing theta-null}

Let $C$ be a smooth projective curve of genus $g$. A \emph{theta characteristic} on $C$ is a line bundle $L$ such that
\[
L^{\otimes 2}\cong K_C.
\]
A \emph{spin curve} is a triple $(C,L,\phi)$ with a specified
isomorphism $\phi:L^{\otimes2}\xrightarrow{\sim}K_C$; we usually suppress
$\phi$ and write $(C,L)$ and $K_C^{1/2}=L$.

Let $\SSS_g$ denote the unrigidified moduli stack of spin curves.
It has even and odd connected components according to the parity of
$h^0(C,L)$. The scalar automorphisms preserving $\phi$ form the central
subgroup $\boldsymbol\mu_2$. Write $\SSS_g^{\mathrm{rig}}$ for its
rigidification. Then
\[
\SSS_g\longrightarrow\SSS_g^{\mathrm{rig}}
\quad\text{is a }\boldsymbol\mu_2\text{-gerbe},\qquad
\SSS_g^{\mathrm{rig}}\longrightarrow\MM_g
\quad\text{is finite \'etale}.
\]
We retain $\SSS_g$ for the unrigidified stack throughout; in particular,
it is the reduced stack of supermoduli used in Section~3.

An even spin curve $(C,L)$ with
\[
h^0(C,L)>0
\]
is said to have a \emph{vanishing theta-null}. For a generic such curve,
\[
h^0(C,L)=2.
\]
In genus $4$, this equality holds for all vanishing theta-nulls.
{Moreover, a vanishing theta-null $L$ on a non-hyperelliptic
genus-$4$ curve $C$ gives a basepoint-free pencil of degree $3$.}
{Indeed, $\deg L=3$, so Clifford's inequality and even parity force
$h^0(L)=2$. A base point would leave a degree-two pencil, contradicting
nonhyperellipticity.}

\subsection{The sheaves \texorpdfstring{$\OO(a,b,c)$}{O(a,b,c)}}

Let $(C,L)$ be a spin curve, with $L=K_C^{1/2}$. Let
\[
p_1,p_2:C\times C\to C
\]
be the projections, and let $\Delta\subset C\times C$ be the diagonal.

\begin{definition}
For integers $a,b,c$, define
\[
\OO(a,b,c)
=
p_1^*K_C^{a/2}\otimes p_2^*K_C^{b/2}\otimes \OO_{C\times C}(c\Delta).
\]
\end{definition}

A local section of $\OO(a,b,c)$ has the form
\[
\varphi=f(x,y)\,\frac{dx^{a/2}dy^{b/2}}{(x-y)^c},
\]
where $f(x,y)$ is holomorphic. Thus $c$ controls the order of pole along the diagonal.

{Let $j_\Delta:C\hookrightarrow C\times C$ be the diagonal embedding.}
Restriction to the diagonal gives the exact sequence
\begin{equation} \label{2.1}
0\to \OO(a,b,c-1)\to \OO(a,b,c)\to {j_{\Delta*}K_C^{(a+b-2c)/2}}\to 0.
\end{equation}

When $a=b$, we use the signed lift of the transposition to the half-form factors, fixed by the convention
\begin{equation} \label{lift}
    dx^{a/2}dy^{b/2}\longmapsto(-1)^{ab}dy^{b/2}dx^{a/2}.
\end{equation}

Together with the interchange of the two factors and of the local equation of $\Delta$, this defines an involution on $\OO(a,a,c)$.
{We take its eigenspaces on cohomology, denoted
$H^i(C\times C,\OO(a,a,c))^\pm$. Since the involution moves the base,
its eigensheaves as modules are formed only after pushforward to the
symmetric square $C^{(2)}$, not as $\OO_{C\times C}$-submodules.}
In particular, for odd $a$ the $+$-eigenspace of sections regular along the diagonal is the exterior, rather than the symmetric, square; for example,
\[
{H^0\bigl(\OO(1,1,0)\bigr)^+}=\bigwedge\nolimits^2H^0(C,L).
\]

\subsection{The Szeg\H{o} kernel}

Let $(C,L)$ be a spin curve with $h^0(C,L)=0$.
Then the Szeg\H{o} kernel, see e.g. \cite{BZB, CV, FI, FKP} is the unique section
\[
s\in {H^0(C\times C,\OO(1,1,1))^+}
\]
whose residue along the diagonal is $1$. Equivalently, it is the unique section mapping to $1\in H^0(\OO_C)$ under the map
\[
{H^0(\OO(1,1,1))^+} \longrightarrow H^0(\OO_C)
\]
coming from \eqref{2.1} with $(a,b,c)=(1,1,1)$.

If $h^0(L)>0$, the Szeg\H{o} kernel itself is no longer defined.
In a one-parameter family $(C_t,L_t)$
degenerating to a vanishing theta-null $(C_0,L_0)$
{at a smooth point of the reduced theta-null divisor with $h^0(C_0,L_0)=2$, and transverse to that divisor,}
the Szeg\H{o} kernel $s_t$ blows up to first order at $t=0$.
The rescaled Szeg\H{o} kernel $ts_t$ has a well-defined limit.
{The zero-mode obstruction map controls the pole order; for the
family used here its nondegeneracy is verified explicitly in
Subsection~{\ref{subsec:relative-normalization}}.}

\section{The codifferential of the super period map}

\subsection{The obstruction sequence on supermoduli}

The exact sequence \eqref{2.1} for $(a,b,c)=(3,3,1)$ gives
\[
{0\to \OO(3,3,0) \to \OO(3,3,1) \to j_{\Delta*}K_C^2 \to 0.}
\]
{Taking the $+$-eigenspaces of global sections, and using
$H^1(C,L^3)=H^0(C,L^{-1})^*=0$, yields}
\begin{equation} \label{obstruction}
0\to \wedge^2 H^0(C,K_C^{3/2})
\to {H^0(C\times C,\OO(3,3,1))^+}
\to H^0(C,K_C^2)\to 0.
\end{equation}
This sequence is the fiber at $(C,L)$ of a sequence on $\SSS_g$ :
\begin{equation} \label{T*}
    0\to \wedge^2 T_-^* \to T^* \to T_+^* \to 0.
\end{equation}
Its extension class was identified in {\cite[Proposition~3.1 and Theorem~3.2]{DW2}} as
the first obstruction to splitting supermoduli space.
Here we will use a related but somewhat different interpretation.

Let $X$ be a supermanifold,
and $i:X_{red} \to X$ the inclusion of the underlying reduced bosonic space.
{The pullback has the canonical parity decomposition}
\begin{equation} \label{pm}
    {i^*T_X=T_+\oplus T_-,}
\end{equation}
where $T_{\pm}$ are the even and odd tangent bundles.
{This decomposition is automatically split; the obstruction is
carried by the different extension \eqref{T*}.}

For the superstack $\FM_g$ of super Riemann surfaces, the bosonic
reduction is the moduli stack $\FM_{g,\mathrm{red}}=\SSS_g$ of spin
curves.  At a spin curve $(C,L)$ the even and odd cotangent spaces are
\[
T^*_{+,(C,L)}\FM_g
=T^*_{(C,L)}\SSS_g
\cong H^0(C,K_C^2),
\qquad
T^*_{-,(C,L)}\FM_g
\cong H^0(C,K_C^{3/2}).
\]
As explained in Section~1.1 of \cite{DW2}, $T_-$ is naturally a vector
bundle on the spin stack, but on the {rigidified spin stack} it is twisted by
the $\boldsymbol\mu_2$-gerbe coming from the scalar automorphism of the spin
bundle.  The tautological involution $\iota$ acts by $-1$ on $T_-^*$ and
therefore by $+1$ on $\bigwedge^2T_-^*$.

Let $X=\FM_g/\iota$, and let $\mathcal J$ be the ideal of
$\SSS_g=X_{\mathrm{red}}$ in $X$.  In {local odd coordinates,}
the first nonzero invariant functions are the quadratic monomials.
Consequently
\[
\mathcal J/\mathcal J^2\cong\bigwedge\nolimits^2T_-^*.
\]
The conormal sequence of $\SSS_g\hookrightarrow X$ is therefore
\[
0\longrightarrow\bigwedge\nolimits^2T_-^*
\longrightarrow\Omega_X\big|_{\SSS_g}
\longrightarrow\Omega_{\SSS_g}\longrightarrow0,
\]
and its fiber at $(C,L)$ is precisely \eqref{obstruction}.  Thus the
middle term denoted $T^*$ in \eqref{T*} is
$\Omega_X|_{\SSS_g}$; the left-hand term is the conormal bundle of the
reduced locus in the quotient, not an odd tangent bundle.

\subsection{The codifferential of the super period map}

The cotangent at the Jacobian of a curve $C$ to
the moduli stack $\Aa_g$ of principally polarized abelian varieties
is identified with the space of quadrics in canonical space, $\Sym^2 H^0(K_C)$.
The cotangent at $C$ to $\MM_g$ is $H^0(K_C^2)$, and the codifferential of the
ordinary period map $p:\MM_g\to \Aa_g$ is the restriction map
\[
d^*p: \Sym^2 H^0(K_C) \to H^0(K_C^2).
\]
\begingroup
We fix the following polarization convention throughout:
\[
P\cdot Q\longmapsto P\boxtimes Q+Q\boxtimes P.
\]
This is \emph{unaveraged} symmetrization; its restriction to the diagonal
is $2PQ$. Thus, under this identification with symmetric kernels, the
displayed restriction map is twice polynomial multiplication. We use
the corresponding scalar normalization of the cotangent identification
in the diagram below. In particular, both rows use the same diagonal
restriction, the right vertical map is the identity, and the numerical
constants in the later kernel products follow this convention.

For a nonhyperelliptic curve $C$, canonical multiplication is surjective,
so the conormal sequence is given, fiber by fiber, by:
\endgroup

\begin{equation}
0\to {H^0(\OO(2,2,-1))^+} \to {H^0(\OO(2,2,0))^+} \to H^0(K_C^2)\to 0.
\end{equation}
Multiplication by the Szeg\H{o} kernel gives the basic commutative diagram:
\begin{equation}
    \begin{tikzcd}[column sep=small]
0 \arrow[r] &
{H^0(\OO(2,2,-1))^+} \arrow[r] \arrow[d,"s_1"] &
{H^0(\OO(2,2,0))^+} \arrow[r] \arrow[d,"s_2"] &
H^0(K_C^2) \arrow[r] \arrow[d,"s_3"] &
0 \\
0 \arrow[r] &
{H^0(\OO(3,3,0))^+} \arrow[r] &
{H^0(\OO(3,3,1))^+} \arrow[r] &
H^0(K_C^2) \arrow[r] &
0.
\end{tikzcd}
\end{equation}

\begingroup
The lower sequence is the conormal sequence of the reduced spin locus
in the bosonic quotient $\FM_g/\iota$. On the open locus $h^0(C,L)=0$,
the vertical maps, given by multiplication with the Szeg\H{o} kernel,
express the codifferential restricted to this reduced locus:
$s_2$ is $d^*\widetilde p|_{\SSS_g^{+,0}}$, and $s_1$ is the induced
map on conormal spaces. The latter records the quadratic odd
contribution, rather than a linear map on odd tangent directions.
Here $\SSS_g^{+,0}$ denotes the open part of $\SSS_g^+$ with $h^0(L)=0$.
\endgroup

This result is based on an analytic formula for the super period matrix
obtained by D’Hoker and Phong in \cite{DP},
explained in its modern formulation by Witten in section 8.3 of \cite{W}.
Codogni and Viviani review this analytic formula in section 6.2 of \cite{CV},
and use it to prove the identification of the codifferential
as multiplication with the Szeg\H{o} kernel in their Theorem 6.5.

This description is made on the open even-spin locus where $h^0(C,L)=0$
(and, globally, after passing to the usual symplectically marked cover).
The super period map itself is not defined by this formula on the
theta-null divisor.  In the degeneration below we use only the
regularized conormal map obtained by multiplying the Szeg\H{o} kernel by
a local equation of that divisor; its special value is a limit, whereas
all assertions about the actual codifferential are made on the nearby
fibers where $h^0(C,L)=0$.

For $g=4$, the target of $s_1$ is
$\wedge^2 H^0(K_C^{3/2})$, where
$\dim H^0(K_C^{3/2})=6.$
Thus a conormal vector determines a $6\times6$ skew form, or equivalently a bivector in
$\wedge^2 H^0(K_C^{3/2})$.
To prove Theorem~\ref{thm:main}, it suffices to show that for a general {even} spin curve of genus $4$,
the image of $s_1$ contains a bivector of rank $6$.

\section{Degeneration to a vanishing theta-null}

Let $(C_t,L_t)$ be a one-parameter family of smooth {nonhyperelliptic} genus $4$ spin curves, chosen transverse to the {reduced} vanishing theta-null divisor {at a smooth point}, with
\begin{itemize}[leftmargin=2em]
\item $h^0(C_t,L_t)=0$ for $t\neq 0$,
\item $h^0(C_0,L_0)=2$.
\end{itemize}
Write
\[
L_0=K_{C_0}^{1/2}.
\]
For $t\neq 0$, let
\[
s_t\in {H^0(C_t\times C_t,\OO(1,1,1))^+}
\]
be the Szeg\H{o} kernel.

The exact sequence
\begin{equation}
    0\to \OO(1,1,0)\to \OO(1,1,1)\to {j_{\Delta*}\OO_C}\to 0
\end{equation}
gives the long exact sequence
\begin{equation}
0\to {H^0(\OO(1,1,0))^+}\xrightarrow{i}
{H^0(\OO(1,1,1))^+}\xrightarrow{j}
H^0(\OO_C)\xrightarrow{k}
{H^1(\OO(1,1,0))^+}.
\end{equation}

Here
\[
{H^0(\OO(1,1,0))^+} \cong \wedge^2 H^0(C,L),
\]
and
\[
{H^1(\OO(1,1,0))^+} \cong H^0(C,L)\otimes H^1(C,L)
\cong \operatorname{End}H^0(C,L),
\]
using Serre duality.

The map $k$ sends $1$ to the identity endomorphism. We thus recover some of the results obtained by
Farkas and Izadi in their beautiful analysis of the theta-null divisor and Szeg\H{o} kernel \cite{FI}:
\begin{itemize}[leftmargin=2em]
\item if $h^0(C,L)=0$, then $k=0$, so $j$ is an isomorphism and the Szeg\H{o} kernel is the unique section with $j(s)=1$;
\item if $h^0(C,L)=2$, then $k$ is injective, so $j=0$ and $i$ is an isomorphism.
\end{itemize}

The preceding long exact sequence describes the special fiber, but by
itself it does not determine the order of the pole in a transverse
family.  Let $\xi{\in H^1(C_0,T_{C_0})}$ be the
Kodaira--Spencer class and let
\[
B_\xi:H^0(C_0,L_0)\longrightarrow H^1(C_0,L_0)
\]
be the first-order obstruction map for lifting the two zero modes.  {Under Serre duality $B_\xi$ is alternating. Near a smooth point
of the reduced theta-null divisor with two zero modes, a local skew
two-by-two block presents the zero-mode cohomology; its Pfaffian is a
local equation of that divisor; {see \cite[Section~4.4]{FKP}}. Transversality therefore makes
$B_\xi$ nonzero and hence invertible. For our explicit family, this
will also follow directly from
$B_\xi(\alpha,\beta)=\pm3/2$ in
Subsection~\ref{subsec:relative-normalization}.}

\begingroup
Here is an algebraic verification that the pole is simple once $B_\xi$
is invertible. The invariant kernel spaces
\[
E_t:=H^0(C_t\times C_t,\OO(1,1,1))^+
\]
have constant dimension $1$ near $t=0$, so they form a line bundle $E$
by cohomology and base change; see also \cite[Proposition~1.2]{FI}.
Choose a local generator $e_t$ and put
$h(t)=\operatorname{Res}_\Delta(e_t)$. Then $h(0)=0$ and
$e_0$ is a nonzero multiple of $\alpha\wedge\beta$, for a basis
$\alpha,\beta$ of $H^0(C_0,L_0)$.

If $h(t)=O(t^2)$, choose a point $z$ where
$e_0(\,\cdot\,,z)\neq0$ and extend it to a section $z_t$ of the family.
Restricting $e_t$ to $C_t\times\{z_t\}$ modulo $t^2$ would have
zero principal part, so it would be a global regular section of
$L_t\otimes(L_t)_{z_t}$ to first order. Trivializing the line
$(L_t)_{z_t}$ would therefore lift the nonzero section
$e_0(\,\cdot\,,z)$ of $L_0$, contradicting the injectivity of $B_\xi$.
Thus $h(t)=t\,v(t)$ with $v(0)\neq0$. For $t\neq0$ the normalized
kernel is $s_t=e_t/h(t)$, so $t\,s_t=e_t/v(t)$ extends holomorphically
with nonzero central value. In particular,
\endgroup

\[
\sigma_t:=t\,s_t
\]
extends holomorphically and has a nonzero special value spanning
\[
{H^0(C_0\times C_0,\OO(1,1,1))^+}
\cong \bigwedge\nolimits^2H^0(C_0,L_0).
\]
Thus, for a basis $\alpha,\beta$ of $H^0(C_0,L_0)$,
\begin{equation}
\label{eq:sigma-zero-general-c}
\sigma_0=c\,\alpha\wedge\beta
\end{equation}
for a nonzero constant $c$.  In the normalized cyclic family used below,
the nondegeneracy of $B_\xi$ and the value of $c$ are computed explicitly
in Section~\ref{subsec:relative-normalization}.

For $t\neq 0$, multiplying the actual Szeg\H{o} kernel by the nonzero scalar $t$ does not change ranks.
{Choose a holomorphic nonvanishing generator $q_t$ of the relative
canonical-quadric line and set}
\[
{\omega_t=q_t\sigma_t\in\wedge^2H^0(C_t,K_{C_t}^{3/2}),}
\]
{where multiplication is understood through the kernel realizations.} Then:
\begin{itemize}[leftmargin=2em]
\item for $t\neq 0$, \(\omega_t\) has the same rank as {the bivector $s_tq_t$;}
\item \(\omega_t\) extends regularly to \(t=0\);
\item on the special fiber, \(\omega_0\) is the rank-$4$ bivector computed in the next section.
\end{itemize}

\section{The rank at theta nulls is 4}

The canonical model of a nonhyperelliptic genus $4$ curve is a complete intersection
\[
C=Q\cap \Gamma \subset \mathbf P^3
\]
of a quadric and a cubic. For a vanishing theta-null, the quadric is singular, hence a cone.

Let $(C_0,L_0)$ be a {nonhyperelliptic} vanishing theta-null of genus $4$, and choose a basis
\[
H^0(C_0,L_0)=\langle \alpha,\beta\rangle.
\]
Choose
\[
a\in H^0(C_0,K_{C_0})
\]
outside the image of
\[
\Sym^2 H^0(C_0,L_0)\subset H^0(C_0,K_{C_0}).
\]
Then
\[
H^0(C_0,K_{C_0})= H^0(C_0,L_0^2) = \langle \alpha^2,\alpha\beta,\beta^2,a\rangle
\]
and
\[
H^0(C_0,K_{C_0}^{3/2})
=  H^0(C_0,L_0^3) =
\langle \alpha^3,\alpha^2\beta,\alpha\beta^2,\beta^3,a\alpha,a\beta\rangle.
\]

\begingroup
{To see that these six elements form a basis of $H^0(C_0,L_0^3)$,
put $U:=H^0(C_0,L_0)$. After choosing $\alpha\wedge\beta$,
the basepoint-free pencil trick gives the short exact sequence}
\[
0\longrightarrow L_0\longrightarrow U\otimes L_0^2
\longrightarrow L_0^3\longrightarrow0.
\]
The kernel on global sections has dimension $2$, so the multiplication
image has dimension $2\cdot4-2=6=h^0(L_0^3)$.
{The products are spanned by $\alpha^3, \dots, a\beta$, which therefore form a basis.}
\endgroup

At the special fiber, the multiplication diagram factors through
\begin{equation}
m:\Sym^2 H^0(C_0,K_{C_0})\otimes \wedge^2 H^0(C_0,L_0)
\longrightarrow
\wedge^2 H^0(C_0,K_{C_0}^{3/2}).
\end{equation}

The singular quadric is
\[
Q_0=\alpha^2\cdot\beta^2-(\alpha\beta)\cdot(\alpha\beta).
\]
Applying $m$ gives
\begin{equation}
m(Q_0,\alpha\wedge\beta)
=
\alpha^3\wedge\beta^3-3\,\alpha^2\beta\wedge \alpha\beta^2.
\end{equation}

This lies in the $4$-dimensional subspace
\[
\Sym^3 H^0(C_0,L_0)
=
\langle \alpha^3,\alpha^2\beta,\alpha\beta^2,\beta^3\rangle
\subset H^0(C_0,K_{C_0}^{3/2}),
\]
and its contraction image is the whole of this subspace, so it has rank $4$.

Thus the limiting bivector $\omega_0$ has rank $4$.

\begin{lemma}\label{lem:Rdual}
Let
\[
V_0=H^0(C_0,K_{C_0}^{3/2}),
\]
and let
\[
R:=\ker\bigl(\iota(\omega_0):V_0^\vee\to V_0\bigr).
\]
Then there is a natural identification, well defined up to an overall scalar,
\[
R^\vee \cong H^0(C_0,L_0).
\]
\end{lemma}

\begin{proof}
Let
\[
\operatorname{Im}\subset V_0
\]
be the image of the contraction map \(\iota(\omega_0)\). Since
\[
{\omega_0=c\bigl(\alpha^3\wedge\beta^3-3\,\alpha^2\beta\wedge\alpha\beta^2\bigr),}
\]
we have
\[
\operatorname{Im}
=
\langle \alpha^3,\alpha^2\beta,\alpha\beta^2,\beta^3\rangle
=
\Sym^3 H^0(C_0,L_0).
\]
Hence
\[
R^\vee = V_0/\operatorname{Im}
\]
is spanned by the classes of \(a\alpha\) and \(a\beta\). Sending these classes to \(\alpha\) and \(\beta\) defines an isomorphism
\[
R^\vee \xrightarrow{\sim} H^0(C_0,L_0).
\]

This is well defined up to an overall scalar. Indeed, replacing \(a\) by
\[
a+\lambda \alpha^2+\mu \alpha\beta+\nu \beta^2
\]
changes \(a\alpha,a\beta\) by elements of \(\operatorname{Im}\), hence does not affect their classes in the quotient. Rescaling \(a\) rescales the resulting map.

Intrinsically,
\[
V_0/\operatorname{Sym}^3H^0(C_0,L_0)
\cong
\bigl(H^0(C_0,K_{C_0})/\operatorname{Sym}^2H^0(C_0,L_0)\bigr)
\otimes H^0(C_0,L_0),
\]
and the choice of the generator \([a]\) of the first factor gives the
stated identification, up to an overall scalar.

\end{proof}

\section{The first-order calculation}
\label{sec:first-order}

The purpose of this section is to prove Proposition~\ref{prop:cyclic-first-variation} below.
At a vanishing theta-null the limiting bivector $\omega_0$ has rank $4$,
with two-dimensional null space $R$.
We will exhibit a curve with vanishing theta-null
and a one-parameter deformation transverse to the theta-null divisor
for which the induced first-order alternating form on $R$ is nonzero.

Write
\begin{equation}
\label{eq:omega-family}
\omega_u=m_u(q_u,\sigma_u),
\end{equation}
where $q_u$ is the quadric through the canonical curve and
$\sigma_u=u s_u$ is the rescaled Szeg\H{o} kernel.
\begingroup
{Here $m_u$ is the multiplication map:}
\[
\begin{aligned}
m_u:\ I_2(K_{C_u})\otimes
H^0(C_u\times C_u,\OO(1,1,1))^+
&\longrightarrow \bigwedge\nolimits^2H^0(C_u,L_u^3),\\
q\otimes\sigma&\longmapsto q\sigma.
\end{aligned}
\]
{obtained from}
\[
\OO(2,2,-1)\otimes\OO(1,1,1)\to\OO(3,3,0)
\]
{by taking global sections}
and using the polarization of Section~3.2.
On the central fiber it agrees with the restriction of the multiplication
map $m$ from Section~5. It does not require the spaces $H^0(C_u,L_u)$
to form a vector bundle.
\endgroup

On the central fiber set
\[
U:=H^0(C_0,L_0),
\qquad
V:=H^0(C_0,K_{C_0}^{3/2}),
\qquad
\operatorname{Im}:=\operatorname{Sym}^3U\subset V.
\]
By Section~5, $\operatorname{Im}$ is precisely the image of the
contraction map defined by $\omega_0$.  Hence, if
\[
R:=\ker\bigl(\iota(\omega_0):V^\vee\longrightarrow V\bigr),
\]
then
\[
R^\vee\cong V/\operatorname{Im}.
\]

\begingroup
Since $H^1(C_u,L_u^3)=0$, the spaces $H^0(C_u,L_u^3)$ form a vector
bundle, {although the $H^0(C_u,L_u)$ jump.} A holomorphic first-order trivialization gives a representative
$\dot\omega_0\in\bigwedge^2V$. We differentiate the total family
$\omega_u=m_u(q_u,\sigma_u)$, including the variation of the
multiplication map as well as that of its arguments. We compute
their combined contribution through the global Gaussian-map
identity
{\eqref{eq:total-gaussian-identity}}
in Subsection~\ref{subsec:relative-first-jet}.
{Injectivity of the relevant part $\overline{\gamma}_0$ of the Gaussian map
follows from Lemma \ref{lem:gaussian-kernel-cyclic}.
This allows us, in Proposition \ref{prop:relative-first-jet},
to determine the intrinsic part of the variation.
This intrinsic part is the first normal symbol:}
\endgroup

\begin{equation}
\label{eq:intrinsic-first-normal-symbol}
\nu:=[\dot\omega_0]
\in
\frac{\bigwedge^2V}{\operatorname{Im}\wedge V}
\cong
\bigwedge\nolimits^2(V/\operatorname{Im})
\cong
\bigwedge\nolimits^2R^\vee.
\end{equation}
Indeed, changing the first-order trivialization changes
$\dot\omega_0$ by $T\cdot\omega_0$ for some $T\in\operatorname{End}(V)$.
Since $\omega_0\in\bigwedge^2\operatorname{Im}$, one has
${T}\cdot\omega_0\in\operatorname{Im}\wedge V$, so the class $\nu$ is independent of
the chosen trivialization.

Our plan is as follows. Subsections~\ref{subsec:cyclic-theta-null}
and~\ref{subsec:quadric-smoothing} specify the central curve and {an}
explicit deformation with its cubic held fixed. We then establish the
Gaussian-map kernel calculation needed to detect the normal quotient.
Proposition~\ref{prop:relative-first-jet} differentiates a global
Gaussian-map identity using
first-order sections of $L_u^3$.
Subsection~\ref{subsec:relative-normalization} computes the polar
coefficient by Serre duality.

\subsection{A cyclic trigonal theta-null}
\label{subsec:cyclic-theta-null}

Let $C$ be a smooth {projective} cyclic trigonal curve of genus $4$ {with affine equation}
\begin{equation}
\label{eq:cyclic-trigonal}
C:\qquad y^3=P(x),
\end{equation}
where $P$ is a square-free polynomial of degree $5$.  There is a unique point $\infty$ above infinity, with
\begin{equation*}
\operatorname{ord}_{\infty}(x)=-3,
\qquad
\operatorname{ord}_{\infty}(y)=-5.
\end{equation*}
Set
\begin{equation*}
L:=\mathcal O_C(3\infty).
\end{equation*}
Then
\begin{equation*}
L^{\otimes 2}\simeq K_C,
\qquad
h^0(C,L)=2,
\end{equation*}
so $L$ is a vanishing theta-null on $C$.

Choose a basis
\begin{equation*}
H^0(C,L)=\langle \alpha,\beta\rangle
\end{equation*}
with $\beta=x\alpha$.  Normalize the identification $L^{\otimes 2}\simeq K_C$ by
\begin{equation*}
\alpha^2=\frac{dx}{y^2},
\end{equation*}
and define
\begin{equation*}
a:=y\alpha^2=\frac{dx}{y}.
\end{equation*}
Then
\begin{equation}
\label{eq:cyclic-H0K}
H^0(C,K_C)
= H^0(C,L^2) =
\langle \alpha^2,\alpha\beta,\beta^2,a\rangle,
\end{equation}
and
\begin{equation}
\label{eq:cyclic-H0K32}
V:=H^0(C,K_C^{3/2})
= H^0(C,L^3) =
\langle
\alpha^3,\alpha^2\beta,\alpha\beta^2,\beta^3,
a\alpha,a\beta
\rangle.
\end{equation}
{These bases follow from the pole orders at infinity. The canonical
map recovers $x$ and $y$ as the ratios of $\alpha\beta$ and $a$ to
$\alpha^2$, respectively. It is therefore birational, so $C$ is
nonhyperelliptic.}

As in Section~5, write
\begin{equation*}
\operatorname{Im}:=\operatorname{Sym}^3H^0(C,L)
=
\langle
\alpha^3,\alpha^2\beta,\alpha\beta^2,\beta^3
\rangle
\subset V.
\end{equation*}
For the rank-$4$ bivector $\omega_0$ of Section~5, Lemma~5.1 gives
\begin{equation*}
R^\vee\simeq V/\operatorname{Im}
=
\langle [a\alpha],[a\beta]\rangle.
\end{equation*}
Hence
\begin{equation}
\label{eq:wedge-null-generator}
\bigwedge\nolimits^2R^\vee
=
\mathbb C\,[a\alpha]\wedge[a\beta].
\end{equation}
The Wronskian of the trigonal pencil is
\begin{align}
\label{eq:w-equals-a2}
w
&:=\alpha\,d\beta-\beta\,d\alpha \\
&=\alpha^2\,dx
=\frac{(dx)^2}{y^2}
=a^2.
\end{align}

\subsection{The canonical quadric and its transverse smoothing}
\label{subsec:quadric-smoothing}

At the theta-null, the canonical quadric is
\begin{equation}
\label{eq:q0-cyclic}
q_0
=
\alpha^2\cdot\beta^2-(\alpha\beta)^2.
\end{equation}
\begingroup
It is a rank-$3$ quadric whose radical in $H^0(C,K_C)^*$ is the line
generated by the covector dual to $a$.

We now fix an explicit smoothing, which will be used in both the
first-jet and the normalization calculations. Write
$P(x)=\sum_{i=0}^5p_ix^i$ and use canonical coordinates
\[
X_0=\alpha^2,\qquad X_1=\alpha\beta,\qquad
X_2=\beta^2,\qquad X_3=a.
\]
Set
\begin{equation}
\label{eq:fixed-homogeneous-cubic}
\begin{aligned}
H={}&p_0X_0^3+p_1X_0^2X_1+p_2X_0^2X_2
      +p_3X_0X_1X_2 +p_4X_0X_2^2+p_5X_1X_2^2, \\
G={}&X_3^3-H,
\end{aligned}
\end{equation}
and keep $G$ fixed in the family
\begin{equation}
\label{eq:explicit-Q-u}
C_u=\{Q_u=G=0\}\subset\mathbf P^3,
\qquad Q_u=X_0X_2-X_1^2+uX_3^2.
\end{equation}
The central complete intersection is the smooth canonical model of
$y^3=P(x)$. After shrinking the parameter disc, all fibers are smooth,
and adjunction gives $K_{C_u}\simeq\OO_{C_u}(1)$, normalized on the
central fiber by $X_0=dx/y^2$. The spin structure $L$ extends over
this disc, since the rigidified spin stack is finite \'etale over
curve moduli. We choose such a spin lift $(C_u,L_u)$.

\begingroup
The geometric meaning of this smoothing can also be seen from the
two rulings of the nearby smooth quadrics. After the base change
$u=r^2$, set
\begin{align*}
X_{00}&=X_0,
&
X_{11}&=X_2,
\\
X_{01}&=X_1+rX_3,
&
X_{10}&=X_1-rX_3.
\end{align*}
Then
\[
Q_{r^2}=X_{00}X_{11}-X_{01}X_{10}.
\]
For $r\neq0$, this is the Segre equation: its two rulings cut out
the two trigonal pencils on $C_{r^2}$. The involution $r\mapsto-r$
interchanges $X_{01}$ and $X_{10}$, transposes the Segre matrix,
and exchanges the two rulings. At $r=0$, both specialize to the
unique ruling of the quadric cone, which cuts out the pencil
$|L|=\mathbf P\langle\alpha,\beta\rangle$ on the central curve.

The symmetric and divided antisymmetric combinations are
\[
S_r:=\frac{X_{01}+X_{10}}{2}=X_1,
\qquad
D_r:=\frac{X_{01}-X_{10}}{2r}=X_3.
\]
The second expression extends regularly across $r=0$, and on
the central fiber we have
\[
S_0=\alpha\beta,\qquad D_0=a.
\]
Thus the equation
\[
X_{00}X_{11}-S_r^2+r^2D_r^2=0
\]
recovers the normalization $u=r^2$ directly from the two rulings.
The explicit family above realizes this normalization with the
chosen central section $a$, and fixes the cubic $G$ for the
subsequent deformation calculations.
\endgroup

In the resulting canonical frame the generator $q_u$ is exactly
\begin{equation}
\label{eq:quadric-normalized-smoothing}
q_u=X_0\cdot X_2-X_1\cdot X_1+uX_3\cdot X_3.
\end{equation}
Here $\dot q_0=a\cdot a$, so its coefficient in the normal direction is
the number
\begin{equation}
\label{eq:lambda-one}
\lambda=1.
\end{equation}
The tensor $a\cdot a\in\operatorname{Sym}^2H^0(C,K_C)$ must be
distinguished from its ordinary product $a^2=w\in H^0(C,K_C^2)$.
More intrinsically, the normal direction is
$\operatorname{Sym}^2(H^0(C,K_C)/\operatorname{Sym}^2H^0(C,L))$;
the tensor $a\cdot a$ represents its generator.
Subsection~\ref{subsec:relative-normalization} verifies directly that
the spin deformation is transverse to the reduced theta-null divisor.

Let $\sigma_u=us_u$ be the rescaled Szeg\H{o} kernel, with central
value $\sigma_0=c\,\alpha\wedge\beta$. Use the kernel line bundle $E$ constructed in Section~4, with
$E_u=H^0(C_u\times C_u,\OO(1,1,1))^+$, and let
$r_u:E_u\to\mathbb C$ be the relative residue map.
In a first-order trivialization of $E$ one has
\begin{equation}
\label{eq:first-residue-jet}
r_u(\sigma_u)=u,\qquad
\dot r_0(\sigma_0)+r_0(\dot\sigma_0)=1,
\qquad r_0=0.
\end{equation}
Thus $\dot r_0(\sigma_0)=1$; there is no global central-fiber kernel
of residue $1$. Indeed, the exact sequence in Section~4 gives
$H^0(C\times C,\OO(1,1,1))^+=\mathbb C(\alpha\wedge\beta)$, and
every section of this space has residue zero. Local derivatives with
a polar part must retain their transition data.

For local diagonal calculations put $\delta=x_1-x_2$. Let
$\varepsilon\in\{1,-1\}$ be the sign fixed by this choice and the
residue convention: after suppressing half-form frames, a local
principal part of residue $1$ is $\varepsilon/\delta$.
Accordingly, the actual relative kernel $\sigma_u$ has local polar
part $\varepsilon u/\delta$. Only $\varepsilon^2=1$ will be needed.
\endgroup

\subsection{The Gaussian map for a general cyclic trigonal curve}
\label{subsec:gaussian-general-cyclic}

The Gaussian map is
\begin{equation}
\label{old:eq:gaussian-map}
\gamma:
\bigwedge\nolimits^2V
\longrightarrow
H^0(C,K_C^4),
\qquad
\gamma(f\wedge g)=f\,dg-g\,df.
\end{equation}
{Here, if $f=F\ell$ and $g=G\ell$ in a local frame $\ell$ of $L^3$,
the expression means $(F\,dG-G\,dF)\ell^2$. Terms involving the
derivative of the frame cancel; the result is a section of
$L^6\otimes K_C=K_C^4$, with no connection chosen.}
In this subsection we compute the kernel of $\gamma$ and conclude that
$\ker(\gamma)\subset\bigwedge\nolimits^2\operatorname{Im}.$

Set
\begin{equation*}
e_0=\alpha^3,
\quad
e_1=\alpha^2\beta,
\quad
e_2=\alpha\beta^2,
\quad
e_3=\beta^3,
\quad
e_4=a\alpha,
\quad
e_5=a\beta.
\end{equation*}
Then
\begin{equation*}
\operatorname{Im}=\langle e_0,e_1,e_2,e_3\rangle,
\qquad
R^\vee=\langle[e_4],[e_5]\rangle.
\end{equation*}
Since $\beta=x\alpha$ and $a=y\alpha^2$, we may rewrite this as:
\begin{equation*}
e_0=\alpha^3,
\quad
e_1=x\alpha^3,
\quad
e_2=x^2\alpha^3,
\quad
e_3=x^3\alpha^3,
\quad
e_4=y\alpha^3,
\quad
e_5=xy\alpha^3.
\end{equation*}

\begin{lemma}
\label{lem:gaussian-kernel-cyclic}
For a general square-free polynomial $P$ of degree $5$, the kernel of
\begin{equation*}
\gamma:\bigwedge\nolimits^2V\longrightarrow H^0(C,K_C^4)
\end{equation*}
is one-dimensional and is generated by
\begin{equation}
\label{eq:gaussian-kernel-generator}
e_1\wedge e_2-\frac13e_0\wedge e_3.
\end{equation}
In particular,
\begin{equation*}
\ker(\gamma)\subset\bigwedge\nolimits^2\operatorname{Im}.
\end{equation*}
\end{lemma}

\begin{proof}
Write $e_i=f_i\alpha^3$, where
\[
(f_0,f_1,f_2,f_3,f_4,f_5)=(1,x,x^2,x^3,y,xy).
\]
On the open set where $x$ is a coordinate, we have:
\[
{\gamma(e_i\wedge e_j)
=\alpha^6(f_i\partial_xf_j-f_j\partial_xf_i)\,dx.}
\]
Factoring out the common nonzero rational $4$-canonical factor
$\alpha^6 dx$, we are reduced to finding linear relations among the coefficient functions
$f_i\partial_xf_j-f_j\partial_xf_i$, in the function field $\mathbb C(C)$.

The only denominators in these coefficient functions arise from
\[
\frac{dy}{dx}=\frac{P'(x)}{3y^2}.
\]
It is therefore convenient to clear them by multiplying every coefficient
by $3y^2$ and to write
\begin{equation*}
{\widetilde\gamma(\eta):=\frac{3y^2\gamma(\eta)}{\alpha^6dx}\in\mathbb C(C).}
\end{equation*}
Here $3y^2$ is invertible in the function field,
so multiplication by it does not alter the linear relations among the images.

The relation $y^3=P(x)$ leads to a useful decomposition of
the function field $\mathbb C(C)$ as a $\mathbb C(x)$-vector space:
\begin{equation}
\label{eq:function-field-cyclic-decomposition}
\mathbb C(C)
=
\mathbb C(x)\oplus y\mathbb C(x)\oplus y^2\mathbb C(x).
\end{equation}
The summands are the three eigenspaces for the cyclic automorphism
$(x,y)\mapsto(x,\zeta y)$, where $\zeta^3=1$ {and $\zeta\neq1$}.

We now compute the images under $\widetilde \gamma$ of the $e_i \wedge e_j$.
For the six wedges involving only $e_0,e_1,e_2,e_3$, one obtains:
\begin{equation*}
\begin{array}{c|c}
\text{bivector} & \widetilde\gamma \\ \hline
 e_0\wedge e_1 & 3y^2 \\
 e_0\wedge e_2 & 6xy^2 \\
 e_0\wedge e_3 & 9x^2y^2 \\
 e_1\wedge e_2 & 3x^2y^2 \\
 e_1\wedge e_3 & 6x^3y^2 \\
 e_2\wedge e_3 & 3x^4y^2
\end{array}
\end{equation*}
Thus these six images span a five-dimensional space, and their unique relation,
consistent with \eqref{eq:gaussian-kernel-generator}, is:
\begin{equation*}
\widetilde\gamma(e_1\wedge e_2)
=\frac13\widetilde\gamma(e_0\wedge e_3).
\end{equation*}
Write $P'(x)$ for the derivative of $P(x)$.  For the eight mixed wedges,
with $0\leq i\leq3$ and $j=0,1$, we use
$e_i=x^i\alpha^3$ and $e_{4+j}=x^jy\alpha^3$.  After suppressing the
common factor as above, the calculation is
\begin{equation}
\label{eq:mixed-gaussian-general-P}
\begin{aligned}
\widetilde\gamma (e_i \wedge e_{j+4})
&=
3y^2\bigl(x^i\,\partial_x(x^jy)-x^jy\,\partial_x(x^i)\bigr) \\
&=
3(j-i)x^{i+j-1}y^3+x^{i+j}P'(x) \\
&=
3(j-i)x^{i+j-1}P(x)+x^{i+j}P'(x),
\end{aligned}
\end{equation}
where a term with exponent $-1$ is absent because its coefficient is zero.  For general $P$, these eight {mixed} polynomials are linearly independent.  Indeed, linear independence is a Zariski-open condition on the coefficients of $P$, so it suffices to show linear independence for one {choice} of $P(x)$. If we take $P(x)=x^5-x+1$, the eight images become:
\begin{equation*}
\begin{array}{c|l}
\text{bivector} & \widetilde\gamma \\ \hline
 e_0\wedge e_4 & 5x^4-1 \\
 e_0\wedge e_5 & 8x^5-4x+3 \\
 e_1\wedge e_4 & 2x^5+2x-3 \\
 e_1\wedge e_5 & 5x^6-x^2 \\
 e_2\wedge e_4 & -x^6+5x^2-6x \\
 e_2\wedge e_5 & 2x^7+2x^3-3x^2 \\
 e_3\wedge e_4 & -4x^7+8x^3-9x^2 \\
 e_3\wedge e_5 & -x^8+5x^4-6x^3,
\end{array}
\end{equation*}
{For this square-free $P$, comparison of coefficients shows that
the eight displayed polynomials are linearly independent.}

Finally,
\begin{equation*}
\widetilde\gamma(e_4\wedge e_5)=3y^4=3P(x)y,
\end{equation*}
which is nonzero.  The three groups of images considered above lie,
respectively, in the three distinct summands
$y^2\mathbb C(x)$, $\mathbb C(x)$, and $y\mathbb C(x)$ of
\eqref{eq:function-field-cyclic-decomposition}.  Thus there are no
relations between the three groups.  Their dimensions are $5$, $8$, and
$1$, respectively, so
\begin{equation*}
\operatorname{rank}(\gamma)=14.
\end{equation*}
Since $\dim\bigwedge^2V=15$, the kernel is one-dimensional, generated by \eqref{eq:gaussian-kernel-generator}.
\end{proof}

\subsection{The relative first-jet calculation}
\label{subsec:relative-first-jet}

\begingroup
For the rest of the proof fix $P(x)=x^5-x+1$, so that
Lemma~\ref{lem:gaussian-kernel-cyclic} applies, and use precisely the
family \eqref{eq:explicit-Q-u}, with its cubic fixed and $\lambda=1$.
In particular, we assert the first-order formula for this specified
family, which suffices for the existence statement below.

\begin{proposition}[Relative first-jet lemma]
\label{prop:relative-first-jet}
For this family, with $\sigma_0=c\,\alpha\wedge\beta$ and the residue
sign $\varepsilon$ fixed in Subsection~\ref{subsec:quadric-smoothing},
the intrinsic first normal symbol is
\begin{equation}
\label{eq:relative-first-jet-formula}
\nu=(3c-\varepsilon)[a\alpha]\wedge[a\beta]
\quad\text{in}\quad\bigwedge\nolimits^2R^\vee.
\end{equation}
\end{proposition}

\begin{proof}
We first establish a global identity before differentiating. For any
$\eta=\sum_i f_i\wedge g_i\in\bigwedge^2H^0(C_u,L_u^3)$, its kernel is
\[
F_\eta(x_1,x_2)=\sum_i
\bigl(f_i(x_1)g_i(x_2)-g_i(x_1)f_i(x_2)\bigr).
\]
This is the K\"unneth identification with
$H^0(C_u\times C_u,\OO(3,3,0))^+$, independent of the expression of
$\eta$ as a sum. Its ordinary anti-invariance makes it vanish on
the diagonal. Thus it is a section of $\OO(3,3,-1)$; its restriction
to the diagonal, under
$\OO(3,3,-1)|_\Delta\simeq K_{C_u}^4$, defines
$J_\Delta(F_\eta)$. Locally it is the coefficient of
$\delta=x_1-x_2$, with the corresponding differential frames restored.
For a decomposable bivector, Taylor expansion at $x_2=x_1-\delta$ gives
\[
F_{f\wedge g}(x,x-\delta)
=-\delta(fg'-gf')+O(\delta^2).
\]
{Therefore, by linearity, for every bivector $\eta$ one has}
\begin{equation}
\label{eq:gaussian-diagonal-jet}
J_\Delta(F_\eta)=-\gamma_u(\eta).
\end{equation}

Let $\mathcal Q_u$ be the symmetric kernel representing $q_u$ under
the unaveraged polarization of Section~3.2. Since it vanishes on the
diagonal and is symmetric, it vanishes there to order at least two.
As a section of $\OO(2,2,-2)$, it restricts on the diagonal to
$\OO(2,2,-2)|_\Delta\simeq K_{C_u}^4$.
This restriction is its second diagonal jet, the global section
\[
J_\Delta^2(\mathcal Q_u)\in H^0(C_u,K_{C_u}^4)
\]
locally represented by $(\mathcal Q_u/\delta^2)|_\Delta$ with frames
restored. Pointwise multiplication of kernels represents $\omega_u$.
For comparison, for $P,Q\in H^0(C,K_C)$ and $v,z\in H^0(C,L)$,
our convention on global decomposable tensors gives
\begin{equation}
\label{eq:multiplication-formula}
m(P\cdot Q,v\wedge z)=Pv\wedge Qz+Qv\wedge Pz.
\end{equation}
The relative kernel $\sigma_u$ has residue $u$ and local polar part
$\varepsilon u/\delta$. Its regular part cannot contribute to the
first jet after multiplication by $\mathcal Q_u$, which vanishes to
second order. Equation~\eqref{eq:gaussian-diagonal-jet} therefore gives
the identity of
global sections
\begin{equation}
\label{eq:total-gaussian-identity}
\gamma_u(\omega_u)=
-J_\Delta(F_{\omega_u})=
-\varepsilon u J_\Delta^2(\mathcal Q_u).
\end{equation}

At $u=0$, set
\[
\mathcal W(x_1,x_2)=
\alpha(x_1)\beta(x_2)-\beta(x_1)\alpha(x_2).
\]
The polarization convention gives
\begin{equation}
\label{q0w2}
\begin{aligned}
\mathcal Q_0
={}&\alpha(x_1)^2\beta(x_2)^2+\beta(x_1)^2\alpha(x_2)^2\\
&-2\alpha(x_1)\beta(x_1)\alpha(x_2)\beta(x_2)
=\mathcal W^2.
\end{aligned}
\end{equation}
Since $(\mathcal W/\delta)|_\Delta=-w$, we obtain
$J_\Delta^2(\mathcal Q_0)=w^2$. The spaces $H^0(C_u,L_u^3)$ and $H^0(C_u,K_{C_u}^4)$ form vector
bundles, since the corresponding $H^1$ groups vanish. The maps
$\gamma_u$ form a holomorphic bundle morphism. Differentiating
\eqref{eq:total-gaussian-identity} in first-order frames of these
bundles yields
\begin{equation}
\label{eq:total-gaussian-derivative}
\gamma_0(\dot\omega_0)+\dot\gamma_0(\omega_0)
=-\varepsilon w^2.
\end{equation}
We must compute the second term, which records the variation of the
Gaussian map itself.

For that purpose work over $\mathbb C[u]/(u^2)$ and construct global
first-order lifts of $e_0,e_1,e_2,e_3$. The charts
$U_0=\{X_0\neq0\}$ and $U_\infty=\{X_2\neq0\}$ cover the family:
if $X_0=X_2=0$, the fixed cubic forces $X_3=0$, and the quadric then
forces $X_1=0$. Use coordinates
\[
x=X_1/X_0,\quad y=X_3/X_0,\qquad
s=X_1/X_2,\quad v=X_3/X_2.
\]
Choose local spin frames $\rho,\tau$ such that
\[
\rho^2=X_0,\qquad \tau^2=X_2,\qquad
\rho|_{u=0}=\alpha,\quad\tau|_{u=0}=\beta.
\]
The two charts are affine, so their central line-bundle frames lift to
the first-order thickenings. Rescaling by square roots of units
congruent to $1$ then gives the required identities for their squares.
On the overlap, modulo $u^2$,
\begin{equation}
\label{eq:spin-transition-first-order}
\tau=x\left(1-\frac{uy^2}{2x^2}\right)\rho,
\qquad
x=\frac{s}{s^2-uv^2},\quad y=\frac{v}{s^2-uv^2}.
\end{equation}
The following table gives global sections $E_i$ of $L_u^3$ to first
order, by their coefficients in these frames:
\begin{equation}
\label{eq:global-first-order-lifts}
\begin{array}{c|c|c}
 E_i &\text{coefficient} \text{ in }\rho^3\text{ on }U_0
 &\text{coefficient in }\tau^3\text{ on }U_\infty\\ \hline
E_0&F_0=1&s^3-\frac32usv^2\\
E_1&F_1=x&s^2-\frac12uv^2\\
E_2&F_2=x^2-\frac12uy^2&s\\
E_3&F_3=x^3-\frac32uxy^2&1
\end{array}
\end{equation}
Indeed the coefficient conversion is
\[
F(x,y,u)\longmapsto
(s^2-uv^2)^{3/2}
F\left(\frac{s}{s^2-uv^2},\frac{v}{s^2-uv^2},u\right),
\]
where the square root specializes to $s$. Expanding modulo $u^2$
gives exactly the right column, whose entries are all regular.
Thus the $E_i$ are global lifts of the four sections of $L^3$,
even though $\alpha,\beta$ themselves do not lift globally as sections
of $L$ in this transverse deformation.

Since $H^1(C,L^3)=0$, we can complete these sections to a first-order frame
$E_0,\ldots,E_5$ by any lifts of $e_4=a\alpha$ and $e_5=a\beta$;
this frame can be lifted holomorphically after shrinking the {base}.
By Section~5, in the central frame
\[
\omega_0=c\Omega_0,\qquad
\Omega_0=e_0\wedge e_3-3e_1\wedge e_2,
\qquad\gamma_0(\omega_0)=0.
\]
In this frame the constant-coordinate lift of $\omega_0$ is
$c(E_0\wedge E_3-3E_1\wedge E_2)$. On the dense open subset where
$x$ is a relative coordinate, the Gaussian formula gives, with
$y_x=\partial_xy$,
{
\[
\begin{aligned}
&(\partial_x F_3) F_0 -(\partial_x F_0) F_3 -3((\partial_x F_2) F_1 - (\partial_x F_1) F_2) \\
&\quad=\partial_x\left(x^3-\tfrac32uxy^2\right)
-3\left\{x\partial_x\left(x^2-\tfrac12uy^2\right)
-\left(x^2-\tfrac12uy^2\right)\right\}\\
&\quad=3x^2-\tfrac32u(y^2+2xyy_x)
-3\left(x^2-uxyy_x+\tfrac12uy^2\right)
=-3uy^2\pmod{u^2}.
\end{aligned}
\]}
All $y_x$ terms cancel.
{The intrinsic version is a formula for the value of $\gamma_u$ in
$H^0(C_u,K_{C_u}^4)$.
As we saw above, the first diagonal jet is obtained by restricting
a section of $\OO(3,3,-1)$ to $\Delta$, with
$\OO(3,3,-1)|_\Delta\simeq K_{C_u}^4$:
the two factors of $L_u^3$ contribute $K_{C_u}^3$,
and the conormal factor $\OO(-\Delta)|_\Delta\simeq K_{C_u}$
contributes $dx$.
Going from our frame-dependent formula above to the intrinsic version
therefore involves two factors of $\rho^3$ and one factor $dx$.
In other words:}
\[
\gamma_u\bigl(c(E_0\wedge E_3-3E_1\wedge E_2)\bigr)
=-3cu\,y^2\rho^6\,dx\pmod{u^2}.
\]
Since the right-hand side already has a factor $u$, extracting its
coefficient of $u$ restricts all remaining factors to the central fiber,
where $\rho|_{u=0}=\alpha$. As $\alpha^2=dx/y^2$, we have
$y^2\alpha^6dx=w^2$, and hence
\begin{equation}
\label{eq:gaussian-map-variation}
\dot\gamma_0(\omega_0)=-3c\,w^2.
\end{equation}
This computation on a dense open set is an equality of global
sections. Variations of the target frame make no contribution because
$\gamma_0(\omega_0)=0$. Combining
\eqref{eq:total-gaussian-derivative} and
\eqref{eq:gaussian-map-variation} gives
\begin{equation}
\label{eq:gaussian-total-normal-image}
\gamma_0(\dot\omega_0)=(3c-\varepsilon)w^2.
\end{equation}

Finally, the induced map
\[
\overline\gamma_0:
\bigwedge\nolimits^2(V/\operatorname{Im})
\longrightarrow H^0(C,K_C^4)/\gamma_0(\operatorname{Im}\wedge V)
\]
is injective. Indeed, if $\gamma_0(\eta)=\gamma_0(\zeta)$ for some
$\zeta\in\operatorname{Im}\wedge V$, then
Lemma~\ref{lem:gaussian-kernel-cyclic} implies
\[
\eta-\zeta\in\ker\gamma_0
\subset\bigwedge\nolimits^2\operatorname{Im}
\subset\operatorname{Im}\wedge V,
\]
so $[\eta]=0$. Moreover,
\[
\gamma_0(e_4\wedge e_5)
=a^2(\alpha\,d\beta-\beta\,d\alpha)=a^2w=w^2.
\]
Projecting \eqref{eq:gaussian-total-normal-image} and applying
injectivity proves \eqref{eq:relative-first-jet-formula}.
\end{proof}
\endgroup

\subsection{The relative normalization}
\label{subsec:relative-normalization}

It remains to determine the constant $c$ in
{\eqref{eq:sigma-zero-general-c}} relative to the normalized smoothing
\eqref{eq:quadric-normalized-smoothing}. We compute it by Serre duality.

\begingroup
Use the fixed-cubic family \eqref{eq:explicit-Q-u} with canonical
coordinates as in Subsection~\ref{subsec:quadric-smoothing}. We retain
the coefficients $p_i$ in the following calculation; it applies in
particular to the chosen $P(x)=x^5-x+1$. On the affine chart $X_0\neq0$, normalized by $X_0=1$, set
$x=X_1/X_0$ and $z=X_2/X_0$.
The homogeneous polynomial $H$ in
\eqref{eq:fixed-homogeneous-cubic} dehomogenizes to
\[
H(x,z)=p_0+p_1x+p_2z+p_3xz+p_4z^2+p_5xz^2,
\qquad H(x,x^2)=P(x).
\]
\endgroup

We compute the Kodaira--Spencer class using the two projective
canonical charts
\begin{equation*}
U_0=\{X_0\neq0\},
\qquad
U_\infty=\{X_2\neq0\}.
\end{equation*}
On $U_0$ put
\begin{equation*}
x=\frac{X_1}{X_0},
\qquad
y=\frac{X_3}{X_0},
\end{equation*}
and on $U_\infty$ put
\begin{equation*}
s=\frac{X_1}{X_2},
\qquad
v=\frac{X_3}{X_2}.
\end{equation*}
{On $U_0$, the quadric equation
\eqref{eq:explicit-Q-u} eliminates $z=X_2/X_0$;
on $U_\infty$, it eliminates $X_0/X_2$. Explicitly,}
\begingroup
\[
z=x^2-uy^2,
\qquad
\frac{X_0}{X_2}=s^2-uv^2.
\]
\endgroup
{Consequently, on the overlap,}
\begin{equation}
\label{eq:projective-transition-u}
x=\frac{s}{s^2-uv^2},
\qquad
y=\frac{v}{s^2-uv^2}.
\end{equation}

On $U_0$, eliminating $X_2/X_0$ from the cubic gives
\begin{equation*}
y^3-P(x)+uy^2R(x)+O(u^2)=0,
\end{equation*}
where
\begin{equation*}
R(x)=p_2+p_3x+2p_4x^2+2p_5x^3.
\end{equation*}
Thus a regular lift of $\partial/\partial u$ on $U_0$ {along the central fiber} is
\begin{equation}
\label{eq:lift-U0}
V_0
=
\frac{\partial}{\partial u}
-\frac{R(x)}{3}\frac{\partial}{\partial y}.
\end{equation}

Similarly, on $U_\infty$ the central fiber is
\begin{equation*}
v^3=P_\infty(s),
\end{equation*}
where
\begin{equation*}
P_\infty(s)
=
p_5s+p_4s^2+p_3s^3+p_2s^4+p_1s^5+p_0s^6.
\end{equation*}
Eliminating $X_0/X_2$ gives to first order
\begin{equation*}
v^3-P_\infty(s)+uv^2S(s)+O(u^2)=0,
\end{equation*}
where
\begin{equation*}
S(s)
=
p_4+p_3s+2p_2s^2+2p_1s^3+3p_0s^4.
\end{equation*}
Hence a regular lift on $U_\infty$ {along the central fiber} is
\begin{equation}
\label{eq:lift-Uinfty}
V_\infty
=
\frac{\partial}{\partial u}
-\frac{S(s)}{3}\frac{\partial}{\partial v}.
\end{equation}

Let $\xi$ denote the Kodaira--Spencer class of the deformation.
Differentiating the transition functions
\eqref{eq:projective-transition-u} and comparing the two regular
lifts \eqref{eq:lift-U0} and \eqref{eq:lift-Uinfty}, then restricting
to $u=0$, gives the \v{C}ech representative
\begin{equation}
\label{eq:KS-cocycle}
V_0-V_\infty
=
-\frac{y^2}{x}\frac{d}{dx}.
\end{equation}
{Here $d/dx$ is the tangent derivation on $y^3=P(x)$, with
$dy/dx=P'(x)/(3y^2)$. The factor $y^2$ makes the vector field
regular at the finite branch points in the overlap.}
For completeness, the equality of the $y$-components is equivalent
to the polynomial identity
\begin{equation*}
-R(x)-\frac{3P(x)}{x^2}
+x^2S(1/x)
=
-\frac{P'(x)}{x},
\end{equation*}
which follows directly from the definitions of $R$ and $S$.

Pairing \eqref{eq:KS-cocycle} with
\begin{equation*}
w=a^2=\frac{(dx)^2}{y^2}
\end{equation*}
gives
\begin{equation}
\label{eq:serre-meromorphic-differential}
w(V_0-V_\infty)
=
-\frac{dx}{x}.
\end{equation}
The curve $C:y^3=P(x)$ has a unique point $\infty$ above infinity,
and
\begin{equation*}
\operatorname{ord}_\infty(x)=-3.
\end{equation*}
Consequently
\begin{equation*}
\operatorname{Res}_\infty\left(-\frac{dx}{x}\right)=3.
\end{equation*}
{For this cover $U_0=C\setminus\{\infty\}$, and the Serre trace
of the meromorphic differential representing the cocycle is, up to
the ordering sign, its residue at the complement of $U_0$.
Thus it is the residue at $\infty$ that computes the pairing.}
Up to the sign convention determined by the ordering of the
\v{C}ech cover,
\begin{equation}
\label{eq:KS-w-pairing}
\langle\xi,w\rangle=\pm3.
\end{equation}

We now relate this pairing to the coefficient $c$ of the polar part
of the Szeg\H{o} kernel. Let
\begin{equation*}
B_\xi:H^0(C,L)\longrightarrow H^1(C,L)
\end{equation*}
be the first-order obstruction map for lifting sections of $L$.
{Using Serre duality, write
$B_\xi(s,t):=\langle B_\xi(s),t\rangle$ for the resulting bilinear form
on $H^0(C,L)$.}
{The form is alternating, as we now verify using the spin identification.}
If the Kodaira--Spencer cocycle is represented locally by
\begin{equation*}
\xi_{ij}=\eta_{ij}(z)\frac{\partial}{\partial z},
\end{equation*}
then its infinitesimal action on a local half-form
$s(z)(dz)^{1/2}$ is
\begin{equation*}
D_{\eta_{ij}}s
=
\eta_{ij}s'+\frac12\eta_{ij}'s.
\end{equation*}
{After pairing with a second section, the symmetric part is the
exact differential
\[
\bigl((D_\eta s)t+(D_\eta t)s\bigr)\,dz=d(\eta st).
\]
Its Serre trace is zero, since the residue of an exact differential is
zero. Thus $B_\xi(s,t)+B_\xi(t,s)=0$.}
Consequently, after pairing with a second section $t$,
antisymmetrization gives
\begin{align*}
B_\xi(s,t)-B_\xi(t,s)
&=
\left\langle
\eta_{ij}(ts'-st'){\,dz},1
\right\rangle\\
&=
\left\langle
\xi,t\,ds-s\,dt
\right\rangle.
\end{align*}
The terms involving $\frac12\eta_{ij}'st$ cancel. Since $B_\xi$
is alternating for a deformation of a theta characteristic,

\begin{equation*}
2B_\xi(s,t)=B_\xi(s,t)-B_\xi(t,s),
\end{equation*}
and hence
\begin{equation}
\label{eq:obstruction-wronskian-exact}
{B_\xi(s,t)
=
\frac12\left\langle\xi,t\,ds-s\,dt\right\rangle
=
-\frac12\left\langle\xi,s\,dt-t\,ds\right\rangle.}
\end{equation}
Applying this to $s=\alpha$ and $t=\beta$, and using
\eqref{eq:KS-w-pairing} together with
\eqref{eq:w-equals-a2}, gives
\begin{equation}
\label{eq:B-alpha-beta}
{B_\xi(\alpha,\beta)
=-\frac12\langle\xi,w\rangle
=\pm\frac32.}
\end{equation}

The same number can be checked directly in the spin frames of
\eqref{eq:spin-transition-first-order}. Local lifts of $\alpha$ are
$\rho$ on $U_0$ and $s\tau$ on $U_\infty$, while local lifts of
$\beta$ are $x\rho$ and $\tau$. Their differences, divided by $u$
and restricted to the central fiber, are respectively
\[
-\frac{y^2}{2x^2}\alpha,
\qquad \frac{y^2}{2x}\alpha.
\]
Pairing the first with $\beta$ gives $-\tfrac12dx/x$, of residue
$3/2$ at infinity; pairing the second with $\alpha$ gives its
negative. The diagonal pairings have residue zero. Thus the
obstruction matrix is alternating and nondegenerate, with
\[
b:=B_\xi(\alpha,\beta)=\pm\frac32.
\]
In particular, this verifies transversality to the reduced theta-null
divisor and the simple pole used above: the local skew zero-mode
block has first term $uB_\xi$, with nonzero linear Pfaffian.

\begingroup
We next fix the scalar relating this obstruction to the polar
coefficient, without an implicit normalization of an inverse Dirac
operator. Choose a point $z\in C$ with $\alpha(z)\neq0$, and extend
it to a section $z_u$ of the family. Restricting the diagonal
principal-parts sequence to $C_u\times\{z_u\}$ gives
\[
0\longrightarrow L_u\otimes(L_u)_{z_u}
\longrightarrow L_u(z_u)\otimes(L_u)_{z_u}
\longrightarrow\mathbb C_{z_u}\longrightarrow0.
\]
Here $\mathbb C_{z_u}$ is the skyscraper sheaf, with its quotient
coordinate normalized by the residue map. The connecting map for
principal part $1$ is Serre dual to evaluation at $z$, up to the
common residue sign. The section $\sigma_u(\,\cdot\,,z_u)$ has
principal part $u$ and central value
\[
c\bigl(\alpha\,\beta(z)-\beta\,\alpha(z)\bigr).
\]
Write $s_z=c(\alpha\,\beta(z)-\beta\,\alpha(z))$, and let
$\partial_z:\mathbb C\to H^1(C,L)\otimes L_z$ be the connecting map
of the central principal-parts sequence. Reducing the relative section
modulo $u^2$ gives the first-order lifting equation
\[
(B_\xi\otimes\operatorname{id}_{L_z})(s_z)=\pm\partial_z(1).
\]
This follows by comparing local regular lifts of $s_z$: their
differences are canceled by $u$ times local lifts of principal part $1$.
Under Serre duality, $\langle\partial_z(1),t\rangle=\pm t(z)$
for $t\in H^0(C,L)$, with the residue normalization of the sequence.
After a trivialization of $(L_u)_{z_u}$, the obstruction on the left is
\[
c\bigl(B_\xi(\alpha)\,\beta(z)
-B_\xi(\beta)\,\alpha(z)\bigr).
\]
Pairing with $\alpha$, and using alternation, gives
$cb\,\alpha(z)=\pm\alpha(z)$. Since $\alpha(z)\neq0$, it follows
that $cb=\pm1$. Therefore
\begin{equation}
\label{eq:c-value}
c=\pm\frac23,\qquad c^2=\frac49.
\end{equation}
\endgroup

Combined with $\lambda=1$, this gives
\begin{equation}
\label{eq:clambda-value}
{(c\lambda)^2=\frac49\neq\frac19.}
\end{equation}
This is the sign-independent numerical statement needed below.

\subsection{Nonvanishing of the first variation}
\label{subsec:first-variation-nonzero}

For the normalized smoothing of
\eqref{eq:quadric-normalized-smoothing}, equation
\eqref{eq:lambda-one} gives
\[
{\lambda=1.}
\]
Proposition~\ref{prop:relative-first-jet} therefore gives
\begin{equation}
\label{eq:final-first-variation-coefficient}
\nu
=
{(3c-\varepsilon)}
[a\alpha]\wedge[a\beta],
\qquad
\varepsilon^2=1.
\end{equation}
If this class vanished, then {$3c=\varepsilon$ and hence $c^2=1/9$.}  This contradicts \eqref{eq:c-value}, which gives
$c^2=4/9$.  Thus $\nu\neq0$.

We have proved the following statement, which will be used in the rank-jump
argument:
\begin{proposition}
\label{prop:cyclic-first-variation}
There exist a smooth cyclic trigonal genus-$4$ curve with a vanishing
theta-null and a one-parameter deformation transverse to the
theta-null divisor such that
\[
\operatorname{pr}_{\wedge^2R^\vee}(\dot\omega_0)=\nu\neq0.
\]
\end{proposition}

\section{The rank jump}

At the special fiber,
\[
\rk \omega_0=4.
\]
Equivalently, the contraction map
\[
\iota(\omega_0):V_0^\vee \to V_0,
\qquad V_0=H^0(C,K_C^{3/2}),
\]
has $2$-dimensional kernel
\[
R:=\ker(\iota(\omega_0))\subset V_0^\vee.
\]

\begin{lemma}\label{lem:rankjump}
If the image of \(\dot\omega_0\) in \(\wedge^2 R^\vee\) is nonzero, then for all sufficiently small \(t\neq 0\),
\[
\rk \omega_t=6.
\]
\end{lemma}

\begin{proof}
Choose a decomposition
\[
V_0^\vee = R\oplus U
\]
such that the restriction of \(\omega_0\) to \(U\) is nondegenerate. Relative to the dual decomposition of \(V_0\), the family \(\omega_t\) has block form
\[
\omega_t=
\begin{pmatrix}
tB+O(t^2) & tC+O(t^2)\\
-tC^t+O(t^2) & J+tD+O(t^2)
\end{pmatrix},
\]
where \(J\) is an invertible \(4\times4\) skew matrix and \(B\) is the skew \(2\times2\) matrix representing the image of \(\dot\omega_0\) in \(\wedge^2 R^\vee\). Since \(\dim R=2\), the condition \(B\neq 0\) implies that \(B\) is invertible. The Schur complement is
\[
{tB+O(t^2),}
\]
whose leading term is \(tB\), hence invertible for \(t\neq 0\) sufficiently small.
\end{proof}

\section{Proof of the main theorem}

\begin{proof}[Proof of Theorem~\ref{thm:main}]
By Proposition~\ref{prop:cyclic-first-variation}, there is a cyclic
trigonal vanishing theta-null and a transverse one-parameter deformation
for which the image of $\dot\omega_0$ in $\wedge^2R^\vee$ is nonzero.
Lemma~\ref{lem:rankjump} therefore gives
$\operatorname{rk}\omega_t=6$ for all sufficiently small $t\neq0$ in
this family.
{For $t\neq0$ these fibers lie in $h^0(C_t,L_t)=0$, and the
associated conormal bivector $s_tq_t$ has rank $6$. Maximal bivector
rank is an open condition. Since the even-spin moduli stack is
irreducible (it is smooth and connected), this gives the assertion for
a general even spin curve of genus $4$. For the ideal noncontainment
below, a single such point already suffices.}

To spell out the last implication, let $f$ be a local equation of the
genus-$4$ Jacobian divisor, and let $q=df$ be its conormal at a general
Jacobian.
{Work on the open supermoduli locus $\FM_4^{+,0}$ fixed in the introduction.}
\begingroup
Since $f$ vanishes on the Jacobian locus, $\widetilde p^{\,*}f$ has
zero reduction. It is even, so in local odd coordinates
$\theta_1,\ldots,\theta_6$ it has the form
\[
\widetilde p^{\,*}f=F_2+F_4+F_6,\qquad
F_2=\sum_{i<j}A_{ij}\theta_i\theta_j,
\]
where $F_k$ is homogeneous of degree $k$ in the odd variables, and $A$ is the skew matrix
representing the quadratic term. The codifferential description
identifies this term with $s_1(q)$, with the fixed normalization of
Section~3.2. Thus $A$ has rank $6$ at the chosen point and
$\operatorname{Pf}(A)\neq0$ there. Products involving an $F_4$ or
$F_6$ in the cube have degree at least $8$, so
\[
\bigl(\widetilde p^{\,*}f\bigr)^3
=F_2^3=3!\operatorname{Pf}(A)\theta_1\cdots\theta_6\neq0.
\]
\endgroup

Thus $f^3\notin\widetilde I_4$, so $(I_4)^3$ is not contained in
$\widetilde I_4$.  On the other hand, the general nilpotence argument
for the six odd variables gives $(I_4)^4\subset\widetilde I_4$.
Consequently the smallest power is $d=4$.
\end{proof}

\end{document}